\documentclass[12pt, reqno]{amsart}
\usepackage{graphicx}
\usepackage{subfigure}
\usepackage{latexsym}
\usepackage{amsmath}
\usepackage{amssymb}

\usepackage[usenames]{color} 
\usepackage{soul} 
\usepackage{pstricks}

\usepackage{amscd}
\usepackage{color}
 \newtheorem{theorem}{Theorem}[section]
 \newtheorem{Def}[theorem]{Definition}
 \newtheorem{Prop}[theorem]{Proposition}
 \newtheorem{Lem}[theorem]{Lemma}
 \newtheorem{Cor}[theorem]{Corollary}

 \newtheorem{Exa}[theorem]{Example}

\def\bd{{\bf d}}
\def\bi{{\bf i}}
\def\bj{{\bf j}}

\def\bu{{\bf u}}
\def\bv{{\bf v}}
\def\pt{\partial}

\newcommand{\D}{{\mathcal D}}
\newcommand{\E}{{\mathcal E}}
\newcommand{\F}{{\mathcal F}}
\newcommand{\Z}{{\mathbb Z}^d}

\date {}

\begin{document}

\title{On the Lipschitz equivalence of self-affine sets}

\author{Jun Jason Luo}
\address{College of Mathematics and Statistics, Chongqing University,  401331 Chongqing, China
\newline\indent Institut f\"ur Mathematik, Friedrich-Schiller-Universit\"at Jena, 07743 Jena, Germany}
\email{jasonluojun@gmail.com}



\thanks{The research is supported in part by the NNSF of China (No.11301322), the Fundamental and Frontier Research Project of Chongqing (No.cstc2015jcyjA00035)}

\begin{abstract}
Let $A$ be an expanding $d\times d$ matrix with integer entries and $\D\subset \Z$ be a finite digit set. Then the pair $(A, \D)$ defines a unique integral self-affine set $K=A^{-1}(K+\D)$. In this paper, by replacing the Euclidean norm with a pseudo-norm $w$ in terms of $A$, we construct a hyperbolic graph on  $(A, \D)$ and show that $K$ can be identified with the hyperbolic boundary. Moreover, if  $(A, \D)$ safisfies  the open set condition, we also prove that two totally disconnected integral self-affine sets are Lipschitz equivalent if an only if  they have the same $w$-Hausdorff dimension, that is, their digit sets have equal cardinality.	We extends some well-known results in the self-similar sets to the self-affine sets.
\end{abstract}

\maketitle

\section{Introduction}

Let $M_d({\mathbb Z})$ denote the class of $d\times d$ matrices with integer  entries and let $A\in M_d({\mathbb Z})$ be expanding, i.e., all its eigenvalues in moduli are strictly bigger than $1$.  Let ${\mathcal D}=\{d_1,\dots, d_N\}\subset \Z$ be a digit set. Define affine maps $S_i(x)=A^{-1}(x+d_i)$ for all $i$. Then $\{S_i\}_{i=1}^N$ forms an {\it iterated function system (IFS)}. There exists a unique nonempty compact set $K$  in ${\mathbb R}^d$ \cite{F} satisfying  
\begin{equation}\label{id-self-affine-set}
K=\bigcup_{i=1}^NS_i(K).
\end{equation}
$K$ is called an (integral) {\it self-affine set}, and a {\it self-similar set} if $A$ is a similar matrix (i.e., $A=nI$ where $n\in {\mathbb N}$ and $I$ is an identity matrix). Usually, we also write it as $$K=A^{-1}(K+{\mathcal D})$$ when we emphasize the affine pair $(A, \D)$. The IFS or the $(A, \D)$ is said to satisfy the {\it open set condition (OSC)} if there exists a bounded nonempty open set $O$ such that $\cup_{i=1}^N S_i(O) \subset O$ and $S_i(O)\cap S_j(O)=\emptyset$ for $i\ne j$.  The  McMullen-Bedford sets (\cite{Be84},\cite{Mc84}) are special cases of such self-affine sets  (see Figure \ref{McMullen sets}).

There are many studies on self-affine sets (see book \cite{F}). Moreover, the related self-affine tiles and tilings are also hot topics in the literature (see the survey paper \cite{Wa99} and references therein).

\begin{figure}[h]
	\centering
	\subfigure[$K$]{
		\includegraphics[width=4cm]{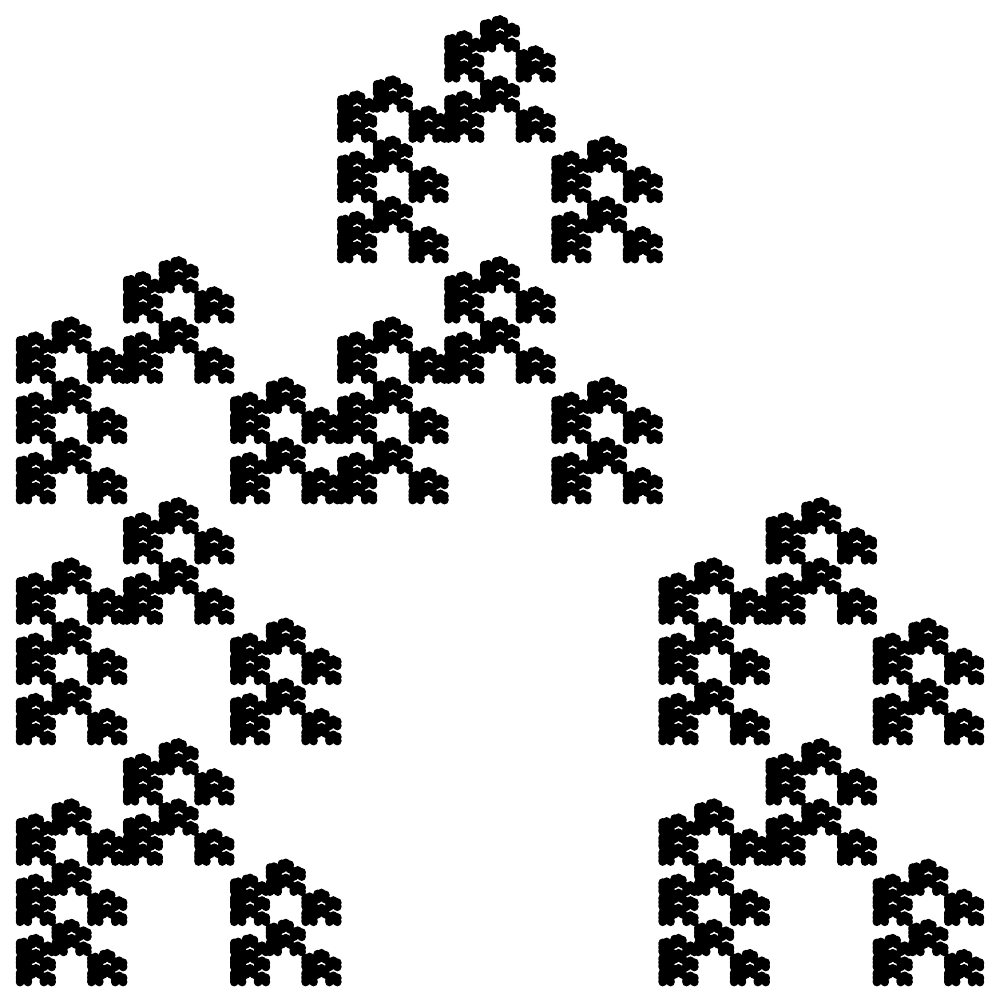}
	}\quad
	\subfigure[$K'$]{
		\includegraphics[width=4cm]{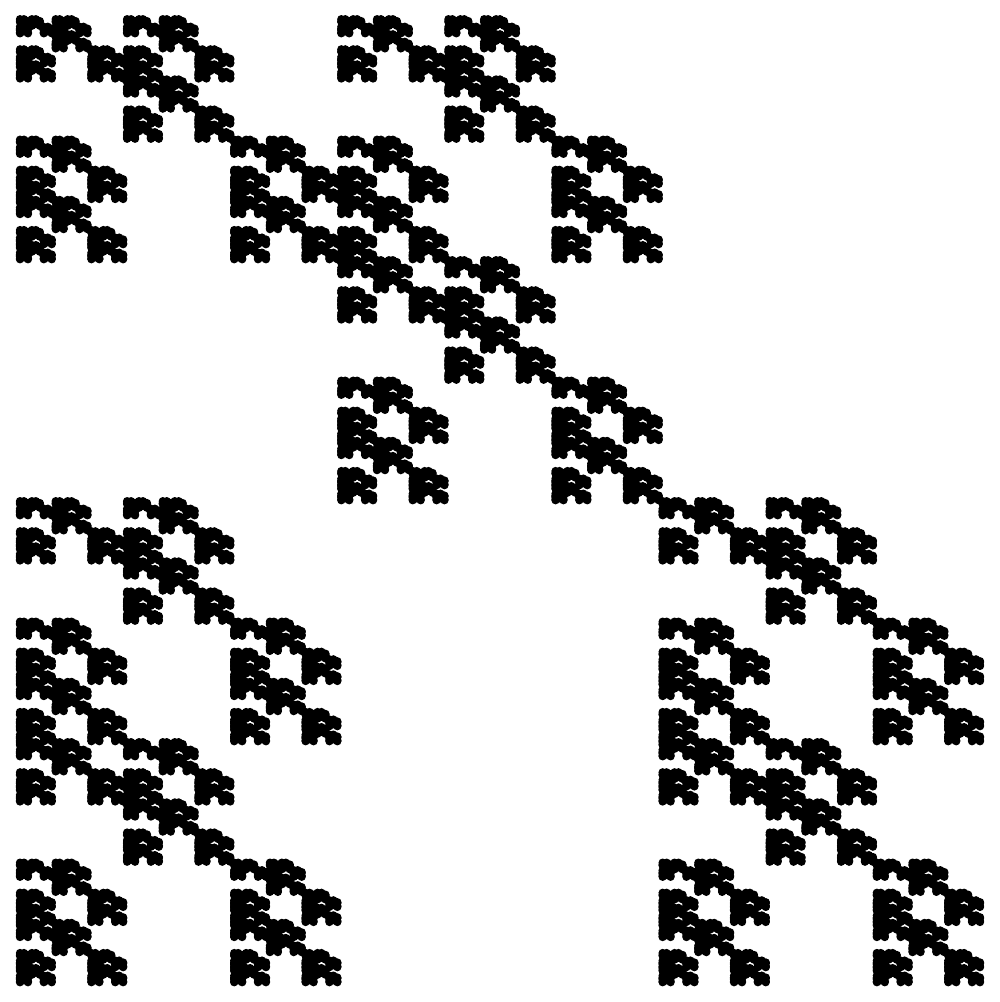}
	} \quad
	\subfigure[$K''$]{
		\includegraphics[width=4cm]{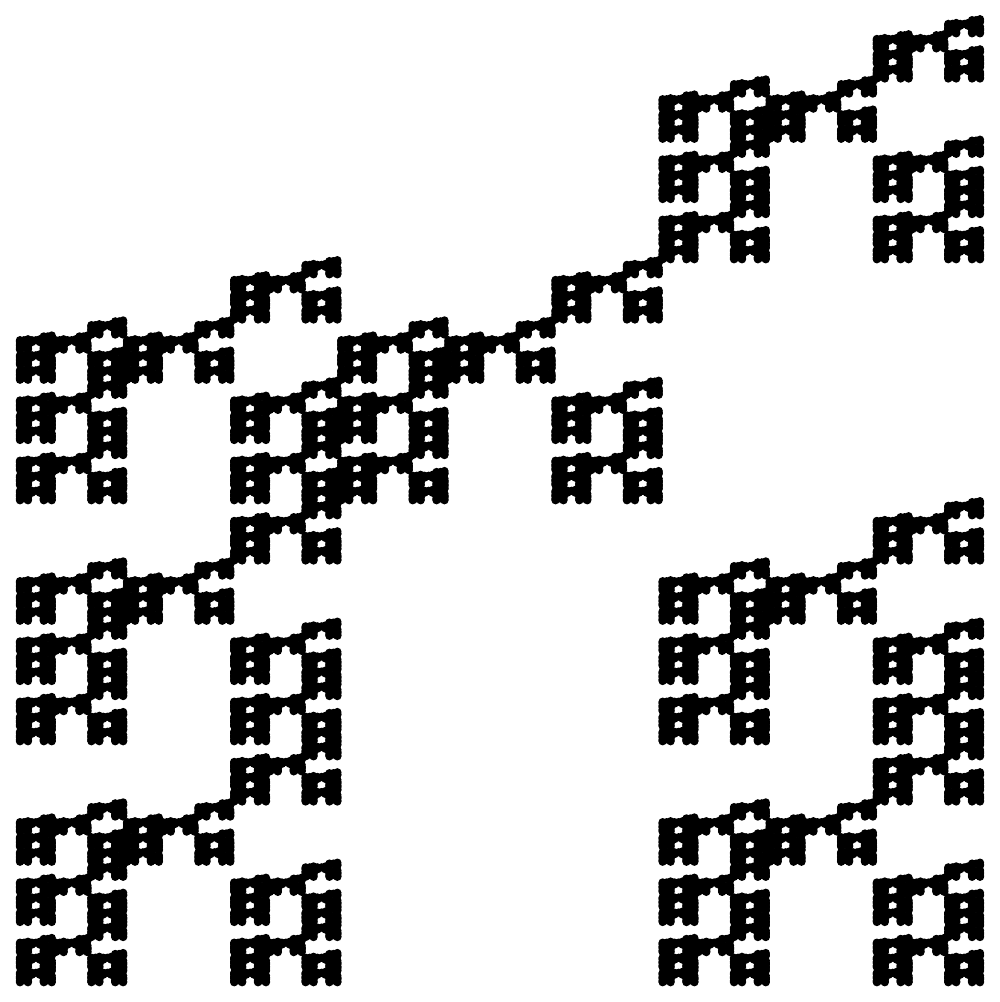}
	}
	\caption{McMullen-Bedford sets with $A=[3,0;0,4]$.}\label{McMullen sets}
\end{figure}

Two metric spaces $(E, d_1)$ and $(F, d_2)$ are said to be {\it Lipschitz equivalent}, denote by $E\simeq F$, if there exists a bi-Lipschitz map $\sigma: E \to F$, i.e., $\sigma$ is a bijection and there exists a constant $C>0$ such that
$$C^{-1}d_1(x,y)\le d_2(\sigma(x), \sigma(y))\le C d_1(x,y), \quad\forall \   x,y\in E.$$
If $E, F\subset {\mathbb R}^d$ and $E\simeq F$, then the above inequality becomes 
\begin{equation}\label{eq-w-equivalence}
C^{-1}\|x-y\|\le \|\sigma(x)- \sigma(y)\|\le C\|x-y\|, \quad\forall \  x,y\in E,
\end{equation}
where $\|\cdot\|$ is the Euclidean norm.  The Hausdorff dimension is an invariant under the bi-Lipschitz map. Like the topological equivalence, Lipschitz equivalence is an important tool for the classification of fractals in fractal geometry and geometric measure theory (\cite{FaMa89},\cite{DaSe97},\cite{Xi04},\cite{LlMa10}). The study of Lipschitz equivalence on Cantor sets  was initiated by Copper and Pignataro \cite{CoPi88} and Falconer and Marsh \cite{FaMa92}. Along this line, it has been undergoing a great development by many people  (\cite{DLL15},\cite{L17},\cite{LaLu13},\cite{RaRuXi06}-\cite{RaZh15},\cite{XiXi10}-\cite{XiXi14}). Up to now, the following is an elegant result on the Lipschitz equivalence of self-similar sets.

\begin{theorem}[\cite{RaRuXi06},\cite{XiXi10},\cite{LaLu13},\cite{XiXi13}]\label{thm0}
	Let $A=nI$ be a similar matrix and  $K=A^{-1}(K+{\mathcal D}_1), K'=A^{-1}(K'+{\mathcal D}_2)$ be two self-similar sets as in \eqref{id-self-affine-set}. If both the IFSs satisfy the OSC and  $K, K'$ are totally disconnected. Then $K\simeq K'$ if and only if $\#{\mathcal D}_1=\#{\mathcal D}_2$.
\end{theorem}

However, due to the complexity and non-uniform contractility from the matrix A, it is difficult to investigate the geometric and topological properties of self-affine sets.  To our knowledge, there are very few results on the Lipschitz equivalence of  self-affine sets. For example, even if $K, K'$ in Figure \ref{McMullen sets} have the same Hausdorff dimension, it is still not clear whether they are Lipschitz equivalent or not.

In order to absorb the non-uniform contractility from $A$, He and Lau \cite{HeLa08} introduced a concept of pseudo-norm $w$  in terms of $A$ (see Section 2) to replace the Euclidean norm and defined the (generalized) $w$-Hausdorff measure  ${\mathcal H}_w^{s}$,  and $w$-Hausdorff dimension $\dim_H^w$. Moreover, they extended Schief's well-known result on self-similar sets \cite{Sc94} to self-affine sets. 

In this paper, we mainly apply the pseudo-norm approach to make an attempt on the Lipschitz equivalence of  self-affine sets.  For distinction, we call $E, F$ {\it $w$-Lipschitz equivalent}, and denote by $E\simeq_w F$ if we replace $\|\cdot\|$ by $w(\cdot)$ in \eqref{eq-w-equivalence}.

On the other hand, Falconer and Marsh in \cite{FaMa92} proposed a {\it nearly Lipschitz equivalence} between $E$ and $F$, denote by $E \simeq_n F$, in the sense that for any  $0<\eta<1$ there exist a bijective map $\sigma: E\to F$ and $C>0$ such that $$C^{-1}\|x-y\|^{1/\eta}\leq \|\sigma(x)-\sigma(y)\|\leq C\|x-y\|^{\eta}, \quad\forall \   x,y\in E.$$ 
The Hausdorff dimension is also an invariant under nearly Lipschitz equivalence. A  relationship between the two kinds of Lipschitz equivalence is as follows.

\begin{Prop}\label{thm2}
	Suppose the eigenvalues of $A$ have  equal  moduli. If $E\simeq_w F$ then $E\simeq_n F$.
\end{Prop}

In studying  the Lipschitz equivalence of self-similar sets, the author  developed a technique of  augmented tree (refer to a series of papers \cite{LaLu13},\cite{DLL15},\cite{L17}). An augmented tree is defined on the symbolic space of the self-similar IFS by adding more edges, and it is a Gromov hyperbolic graph (\cite{LaWa09},\cite{LaWa16}). 

Now under the setting of self-affine IFS as in \eqref{id-self-affine-set}, we let $\Sigma =\{1, \dots, N\}$ and $\Sigma^* :=  \bigcup_{n=0}^\infty \Sigma^n$ where $\Sigma^0=\emptyset$. For $\bu=i_1\cdots i_n\in \Sigma^*$, write the length $|\bu|=n$ and the composition $S_\bu=S_{i_1}\circ\cdots \circ S_{i_n}$. Say $\bu, \bv\in X$ are equivalent, denote by $\bu\sim\bv$, if $S_{\bu}=S_{\bv}$. Then $\sim$ determines an equivalence relation on $\Sigma^*$. Set $X_n=\Sigma^n/\sim$  for each $n$. Then $X=\bigcup_{n=0}^\infty X_n$ is the quotient space of $\Sigma^*$ and $[\bu]$ the equivalence class. For convenience, we still use $\bu\in X$ to replace $[\bu]\in X$ with no confusions.

There is a natural graph structure on $X$ by the standard concatenation of words (see details in Section 2), we denote the edge set by $\E_v$.  Let $J$ be a nonempty bounded closed invariant set of the IFS, i.e., $S_i(J) \subset J$ for each $i$.  We define   horizontal edges on the graph $(X, \E_v)$ by 
\begin{equation*}
\E_h=\bigcup_{n=1}^\infty\{(\bu,\bv)\in X_n\times X_n:  \bu\ne \bv, \text{ and } S_\bu(J)\cap S_\bv(J)\ne \emptyset\}.
\end{equation*}
Let  $\E=\E_v\cup \E_h$, then the graph $(X, \E)$ resembles the augmented tree (see Definition \ref{Def'}).

We use the standard notation on hyperbolic graph $X$ introduced by Gromov (\cite{Gr87}, \cite{Wo00}). The hyperbolic boundary is defined by $\partial X=\hat{X}\setminus X$ where $\hat{X}$ is the completion of $X$ under a visual metric $\rho_a$ on $X$ (see Section 2). The following main result is a generalization of the self-similar case (\cite{LaWa09},\cite{Wa14},\cite{LaWa16}).

\begin{theorem}\label{thm1}
	Let $K$ be the integral self-affine set as in \eqref{id-self-affine-set} with $|\det A|=q$. Then the graph $(X, \E)$ is hyperbolic. Moreover, $K$ is H\"older equivalent to the hyperbolic boundary of $(X, \E)$, i.e., there exists a bijective map $\phi: \partial X \to K$ such that 
	\begin{equation*}
	C^{-1}w(\phi(\xi)-\phi(\eta)) \le \rho_a(\xi,\eta)^\alpha \le Cw(\phi(\xi)-\phi(\eta)), \quad \forall \  \xi\ne \eta \in \partial X
	\end{equation*}
	where $\alpha=\frac{\log q}{da}$ and $C>0$ is a constant.
\end{theorem}

The theorem together with Theorem 1.1 of \cite{DLL15} (see Theorem \ref{th1.1}) helps us extend Theorem \ref{thm0} to the framework of self-affine sets.

\begin{theorem}\label{thm3}
	Let $K=A^{-1}(K+{\mathcal D}_1)$ and $K'=A^{-1}(K'+{\mathcal D}_2)$ be two  integral self-affine sets. Suppose both the IFSs satisfy the OSC and $K, K'$ are totally disconnected. Then $K\simeq_w K'$ if and only if $\#{\mathcal D}_1=\#{\mathcal D}_2$.
\end{theorem}

As a corollary, if we further assume that the eigenvalues of $A$ have  equal  moduli then $K, K'$ are nearly Lipschitz equivalent if and only if $\#{\mathcal D}_1=\#{\mathcal D}_2$ (Corollary \ref{cor-iff}). Moreover, we shall show in Example \ref{exa} that the  McMullen-Bedford sets in Figure \ref{McMullen sets} are $w$-Lipschitz equivalent.

In general, the IFS in \eqref{id-self-affine-set} has overlaps. It actually satisfies the weak separation condition (\cite{LaNgRa01}) which is weaker than the OSC. Hence if we remove the OSC in the assumption, then the theorem would be false. As for the overlapping case,  we need more deep discussions  (see \cite{L17}).

For the organization of the paper, we recall basic results on hyperbolic graph and pseudo-norm in Section 2. In Section 3, we construct a hyperbolic graph  on the affine pair $(A, \D)$ and prove Theorem \ref{thm1}. In Section 4, we first prove Proposition \ref{thm2}, and then prove  Theorem \ref{thm3} by some technical lemmas.

\bigskip
\section{Hyperbolic graph and pseudo-norm}
Let $X$ be an infinite connected graph. For $x, y \in X$, let  $\pi(x, y)$ denote a geodesic from $x$ to $y$, and $d(x,y)$ its length. Fix a vertex $o$ as a reference point of $X$, and let $|x| = d(o, x)$. The degree of a vertex $x$ is the total number of edges connecting to $x$ and is denoted by $\deg(x)$. A graph is said to be of {\it bounded degree} if $\max\{\deg(x): x\in X\}<\infty$.  According to \cite{Wo00}, for $x, y \in X$, let
\begin{equation*}
	|x\wedge y|=\frac12\big(|x|+|y|-d(x,y)\big)
\end{equation*}
denote the {\it Gromov product}, and  call $X$  {\it hyperbolic}  if there is $\delta \geq 0$ such that
$$
|x\wedge y|\ge \min\{|x\wedge z|,|z\wedge y|\}-\delta \quad \forall \  x,y,z\in X.
$$
For $a>0$ with $\exp(3\delta a) < \sqrt 2$, we define a binary map on $X$ by
\begin{equation*}
	\rho_a(x,y)=\delta_{x,y}\exp(-a|x\wedge y|),
\end{equation*}
where $\delta_{x,y}=0,1$ according to $x=y$ or $x\ne y$. The $\rho_a$ is not necessarily a metric, but it is equivalent to a metric (\cite{Wo00}). Hence we always regard $\rho_a$ as a {\it visual metric} for convenience. Let $\hat{X}$ be the $\rho_a$-completion of $X$. We call $\pt X= \hat{X}\setminus X$ the {\it hyperbolic boundary} of $X$.  It is clear that $\rho_a$ can be extended to $\pt X$,  and  $\pt X$ is a compact set under $\rho_a$. It is useful to identify $\xi\in\pt X$ with a {\it geodesic ray} in $X$ that converges to $\xi$.

Let $X$ be a tree  (i.e., any two distinct vertices can be connected by only one path). Trivially, $X$ is hyperbolic (with $\delta =0$), and the hyperbolic boundary is a Cantor set. We use $\E_v$ to denote the set of edges of $X$ ($v$ for vertical), and $X_n = \{x \in X: |x| =n\}$ the $n$-th level of $X$. We introduce some additional edges on each level of $X$.

\begin{Def}[\cite{Ka03},\cite{LaWa09}]  \label{Def'} 
	Let $(X, \E_v)$ be a tree. We call $(X, \E)$ an augmented tree if $\E=\E_v\cup \E_h$,  where $\E_h\subset (X\times X)\setminus \{(x,x):\ x\in X\}$ is symmetric and satisfies
	\begin{equation*}
		(x,y)\in \E_h\Rightarrow |x|=|y|, \mbox{ and either } x^{-}=y^{-} \mbox{ or } (x^{-},y^{-})\in \E_h.
	\end{equation*}
	($x^{-}$ is the {\it predecessor} of $x$.) We call elements in $\E_h$ horizontal edges.
	
	Furthermore, if each vertex of $X$ has $N$ offspring, we call  $(X, \E_v)$ an $N$-ary tree and   $(X, \E)$ an $N$-ary augmented tree.
\end{Def}

For $x, y \in X$, the geodesic path of $x, y$ is not unique in general, but there is a canonical one of the form
\begin{equation*}
	\pi (x, y) = \pi(x, u) \cup \pi (u,v) \cup \pi (v,y)
\end{equation*}
where $\pi (x, u), \pi (v, y)$ are vertical paths, $\pi (u, v)$ is a horizontal path,  and for any geodesic $\pi' (x,y)$, $d(o, \pi (u,v)) \leq d(o, \pi' (x,y))$.  It can happen that there are only two parts with $v=y$ or $x = u$.  For a canonical geodesic $\pi(x,y)$, the  Gromov product can be written as  $$|x\wedge y|=h-\ell/2$$ where $h, \ell$ are the level and the length of the horizontal part $\pi(u,v)$, respectively.

Let  $A\in M_d({\mathbb Z})$ be  expanding  with $|\det A|=q$. Let $B_\delta:=B(0,\delta)$ be the open ball in ${\mathbb R}^d$ with center at $0$ and radius $\delta$. Following \cite {HeLa08}, we introduce the notion of pseudo-norm (with respect to $A$) as follows: For $0<\delta< 1/2$, let $\varphi \geq 0 $ be a $C^\infty$ function supported in   $B_\delta$  with $\varphi (x)
=\varphi (-x)$ and $\int_{{\mathbb R}^d} \varphi =1$. Let $V=AB_1\setminus B_1$, and let $h = \chi_V\ast \varphi$ be the convolution of the indicator function $\chi_V$ and  $\varphi$. We define
$$
w(x) = \sum_{n=-\infty}^\infty q^{-n/d}h(A^nx), \quad x \in {\mathbb R}^d.
$$ 
Then $w(x)$ satisfies  
\begin{enumerate}
	\item $w(x) = w(-x)$ and $w(x) =0$ if and only if $x=0$, 
	
	\item $w(Ax) = q^{1/d}w(x)$, and 
	
	\item there exists $\beta>0$ such that $w(x+y) \leq \beta \max \{w(x), w(y)\}$. 
\end{enumerate}
The $w$ is used as a distance (ultra-metric) to replace the Euclidean
distance to define $\hbox {diam}_w(E)=\sup\{w(x-y):x,y\in E\}$ (the diameter of a set $E$) and $\hbox{dist}_w(E,F)=\inf\{w(x-y): x\in E, y\in F\}$ (the distance between sets $E$ and $F$). Moreover, the  $w$-Hausdorff distance $d^w_H$,  $w$-Hausdorff measure ${\mathcal H}_w^{s}$,  and  $w$-Hausdorff dimension $\dim_H^w$ are also well-defined accordingly.

The new and old definitions of norm and dimension have a simple relationship through $\lambda_1$ and $\lambda_0$, the maximal and minimal moduli of the eigenvalues of $A$.

\begin{Prop}[\cite{HeLa08}]\label{prop-norms}
	(i) If $\|x\|\le 1$. Then for any $0<\epsilon<\lambda_0-1$, there exists a constant $C_\epsilon>0$ such that  
	\begin{align*}
	C_\epsilon^{-1}\|x\|^{\ln q/{d\ln(\lambda_0-\epsilon)}}\le w(x) \le C_\epsilon\|x\|^{\ln q/{d\ln(\lambda_1+\epsilon)}}.
	\end{align*}
	
	(ii) For any subset $E$ of ${\mathbb R}^d$, we have $$\frac{\ln q}{d\ln \lambda_1}\dim_H^w E\le \dim_H E\le\frac{\ln q}{d\ln \lambda_0}\dim_H^w E.$$
\end{Prop}

Under the  pseudo-norm, most of the basic properties for the
self-similar sets (including Schief's basic result on the OSC) can be
carried to the self-affine sets and graph-directed sets (see \cite {HeLa08} and \cite{LuYa10}). 

\begin{theorem}[\cite{HeLa08}]\label{th-dimformular}
	Let $A$ be an expanding matrix with $|\det A|=q$ and ${\mathcal D}=\{d_1, \dots, d_N\}\subset {\mathbb R}^d$. Let  $K=A^{-1}(K+{\mathcal D})$ be the associated self-affine set. If $(A, \D)$ satisfies the OSC, then $\dim^w_H K=d\ln N/\ln q:=s$ and $0<{\mathcal H}_w^{s}(K)<\infty$.
\end{theorem}

\bigskip
\section{Hyperbolic graph induced by $(A, \D)$}

Let the IFS $\{S_i\}_{i=1}^N$ and the self-affine set $K$ be as in \eqref{id-self-affine-set}. Now we construct a graph structure on the symbolic space that represents the IFS.  Let $\Sigma =\{1, \dots, N\}$ and $\Sigma^* :=  \bigcup_{n=0}^\infty \Sigma^n$ be the symbolic space where $\Sigma^0=\emptyset$ (as a reference point). For $\bu=i_1\cdots i_n\in \Sigma^*$, write the length $|\bu|=n$ and the composition $S_\bu=S_{i_1}\circ\cdots \circ S_{i_n}$. Say $\bu, \bv\in X$ are equivalent, denote by $\bu\sim\bv$, if $S_{\bu}=S_{\bv}$. Then $\sim$ defines an equivalence relation on $\Sigma^*$. Let $X_n=\Sigma^n/\sim$ for each $n$. Then $X=\bigcup_{n=0}^\infty X_n$ is the quotient space of $\Sigma^*$ and $[\bu]$ is the equivalence class containing $\bu$. For convenience, we still use $\bu\in X$ to replace $[\bu]\in X$ with no confusions.

There is a natural graph structure on $X$ by the standard concatenation of finite words, we denote the edge set by $\E_v$. That is, $(\bu, \bv)\in \E_v$ if and only if there exist $\bi\in [\bu], \bj\in [\bv]$ and some $\ell\in \Sigma$ such that $\bj=\bi\ell$ or $\bi=\bj\ell$.  We notice that if the IFS satisfies the OSC, then $S_{\bu}\ne S_\bv$ for any distinct $\bu, \bv$. Hence $X=\Sigma^*$ and $(X,\E_v)$ is an $N$-ary tree.

Let $J$ be a  closed invariant set of the IFS, i.e., $S_i(J) \subset J$ for all $i$. For  $\bu=i_1\cdots i_n\in \Sigma^*$, we let  $J_{\bu}:=S_{\bu}(J)=A^{-n}(J+d_{\bu})$ where $d_{\bu}=d_{i_n}+Ad_{i_{n-1}}+\cdots+A^{n-1}d_{i_1}$. According to the geometry of $K$, we define more (horizontal) edges on $X$:
\begin{equation*}
\E_h=\bigcup_{n=1}^\infty\{(\bu,\bv)\in X_n\times X_n:  \bu\ne \bv, \text{ and } J_\bu\cap J_\bv\ne \emptyset\}.
\end{equation*}

Let $\E=\E_v\cup \E_h$, then the graph $(X, \E)$ resembles the augmented tree in Definition \ref{Def'} by the observation: As $J_\bu\subset J_{\bu^-}, J_\bv\subset J_{\bv^-}$, if $(\bu,\bv)\in \E_h$ then either $\bu^-=\bv^-$ or $(\bu^-,\bv^-)\in \E_h$.

\begin{Lem}[\cite{Wa14}] \label{lem-criterion of hyperbolicity}
	The graph $(X, \E)$ is hyperbolic if and only if  the lengths of  horizontal geodesics are uniformly bounded.
\end{Lem}

In particular, if the IFS satisfies the OSC, then  the graph $(X, \E)$ indeed is an $N$-ary augmented tree which has been studied in detail in \cite{LaWa09},\cite{LaLu13},\cite{DLL15} and \cite{LaWa16}. 

The invariant set $J$ can be quite flexible, for example we can take $J=K$, or take $J=\overline{U}$ for the open set $U$ in the OSC, or take $J$ to be some sufficiently large closed ball. The graph $(X,\E)$ depends on the choice of $J$. But under our IFS as in \eqref{id-self-affine-set}, the hyperbolic boundary is the same as they can be identified with the underlying self-affine set (see  Theorem \ref{th-holderequivalence}).

Now we fix  $J=\overline{B_\delta}$ (the closure of a ball $B_\delta$). For $T\subset X$, we say  that $T$ is a {\it horizontal component} if  $T\subset X_n$ for some $n$ and $T$ is a maximal connected subset of $X_n$ with respect to ${\E}_h$.  Write $J_T:=\bigcup_{\bu\in T}S_{\bu}(J)$. Geometrically, $T$ is a horizontal component of $X$ if and only if  $J_T$ is a connected component of $\bigcup_{\bu\in X_n}S_{\bu}(J)$. 

Let  $T\subset X_n, T'\subset X_m$ be two horizontal components of $X$.  We  say  that $T$ and $T'$ are {\it equivalent}, denote by $T \sim T'$, if there exists an affine map 
$$g(x)=A^{n-m}x+\bd,  \quad\text{where }  \bd\in \Z$$  such that 
$\{g\circ S_\bu: \bu\in T\}=\{S_{\bu'}: \bu'\in T'\}$.
Obviously, if $T \sim T'$, then $\#T = \#T'$ and $g(J_{T})=J_{T'}$. Denote by $[T]$ the equivalence class and $\F$  the family of all horizontal components of  $X$.

\begin{Prop}
	Let $T, T'\in \F$, and let $\{T_1,\dots, T_n\}, \{T_1',\dots, T_{n'}'\}\subset \F$ be the horizontal components that consist  of offspring of $T, T'$ respectively. Suppose $T\sim T'$. Then $$\{[T_i]:1\le i\le n\}=\{[T_i']:1\le i\le n'\}$$ counting multiplicity. In particular, $n=n'$.
\end{Prop}

\begin{proof}
	Since $T\sim T'$, without loss of generality, we assume that $T=\{\bu_1,\dots, \bu_m\}, T'=\{\bu'_1,\dots, \bu'_m\}$ and $g\circ S_{\bu_i}=S_{\bu'_i}$ where $1\le i\le m$ and $g$ is an affine map as in the definition. Then for any  $1\le i\le m$ and  $j\in \Sigma$,   $$g\circ S_{\bu_i j}=g\circ S_{\bu_i}\circ S_j=S_{\bu_i'}\circ S_j=S_{\bu_i'j}.$$ Hence  $$\{g\circ S_{\bu_ij}: \bu_i\in T, j\in \Sigma\}=\{S_{\bu_i'j}: \bu_i'\in T', j\in \Sigma\}.$$
	
	It follows that $S_{\bu_ik}=S_{\bu_j\ell}$ if and only if $S_{\bu_i'k}=S_{\bu_j'\ell}$ and $S_{\bu_ik}(J)\cap S_{\bu_j\ell}(J)\ne\emptyset$ if and only if $S_{\bu_i'k}(J)\cap S_{\bu_j'\ell}(J)\ne\emptyset$,  completing the proof.
	\end{proof}

\begin{Def}\label{def of simple tree}
	We call the graph $(X, \E)$  {\it simple} if the equivalence classes in $\F$ is finite, that is, $\#(\F/\sim)<\infty$.
\end{Def}

We remark that the definition of simple graph is slightly stronger than the original one in \cite{DLL15} which is defined from the graphical point of view. Hence under the OSC (where the graph $(X, \E)$ becomes an $N$-ary augmented tree), we have

\begin{theorem} [\cite{DLL15}]\label{th1.1}
	Suppose an $N$-ary augmented tree $(X, {\mathcal E})$ is simple, then 
	
	(i) $(X,\E)$ is a hyperbolic graph;
	
	(ii) $\partial(X, {\mathcal E})\simeq\partial(X, {\mathcal E}_v)$, which is an $N$-Cantor set.
	\end {theorem}

\begin{Lem}\label{h-condition}
Let $K$ be an integral self-affine set as in \eqref{id-self-affine-set}. Then for any bounded closed invariant set $J$, there exist $c>0$ and $k\ge 0$ such that for any $n\ge 1$ and ${\bf u,v}\in \Sigma^n$, 
$$J_\bu\cap J_\bv=\emptyset \implies \hbox{dist}_w(J_{\bf ui}, J_{\bf vj})\ge cq^{{-n}/d},  \quad\forall \  {\bf i,j}\in \Sigma^k.$$
\end{Lem}

\begin{proof}
	We first claim that there exists $c'>0$ such that for any integer $n\ge 1$ and $\bu,\bv\in \Sigma^n$,  $$K_\bu\cap K_\bv=\emptyset \implies \hbox {dist}_w(K_\bu, K_{\bv})\ge c'q^{{-n}/d}.$$
Indeed let $$c'=\inf\{w(x): x\in\cup\{K-K+d: d\in {\mathbb Z}^d \text{ such that } 0\notin K-K+d\}\}.$$ Then $c'>0$ as $K$ is compact. If $K_\bu\cap K_\bv=\emptyset$, then for any $x,y\in K$, we have 
$$w(S_\bu(x)-S_\bv(y))=w(A^{-n}(x-y+d_\bu-d_\bv))=q^{{-n}/d}w(x-y+d_\bu-d_\bv)>0.$$ By making use of $d_\bu, d_\bv\in {\mathbb Z}^d$ and the above expression of $c'$, hence we have $$w(S_\bu(x)-S_\bv(y))\ge c'q^{{-n}/d}.$$

For the  invariant set $J$,  we have $K\subset J$ and the $w$-Hausdorff distance $d_H^w(K_\bi, J_\bi)\le c_1q^{{-k}/d}$ for all $\bi\in \Sigma^k$. In particular we take $k$ so that $c_1q^{{-k}/d}< c'/3$.

If $J_\bu\cap J_\bv=\emptyset$, then $K_\bu\cap K_\bv=\emptyset$, it follows that $\hbox {dist}_w(K_\bu, K_{\bv})\ge c'q^{{-n}/d}.$  Applying this to the level $n+k$, we have 
$$
\begin{aligned}
\hbox {dist}_w (J_{\bf ui}, J_{\bf vj}) & \geq  \hbox {dist}_w (K_{\bf ui}, K_{\bf vj}) -  d_H^w(K_{\bf ui}, J_{\bf ui})- d_H^w(K_{\bf vj}, J_{\bf vj})\\
&  \geq
 c'q^{{-n}/d} -(2c'/3) q^{{-n}/d}\geq (c'/3)q^{{-n}/d}  \qquad \forall \ {\bf i},\  {\bf j} \in \Sigma^k.
\end{aligned}
$$
The lemma follows by taking $c = c'/3$.
\end{proof}

\begin{Lem}\label{lem-WSC}
	Let $\{S_i\}_{i=1}^N$ be the IFS as in \eqref{id-self-affine-set}. Then for any $b>0$, there exists a constant $\gamma :=\gamma(b)$ such that for any set $V\subset {\mathbb R}^d$ with $\hbox {diam}_w(V)\le b$, $$\#\{\bu\in X: \  (J+d_\bu)\cap V\ne\emptyset\}\le\gamma,$$
	where $d_{\bu}=d_{i_n}+Ad_{i_{n-1}}+\cdots+A^{n-1}d_{i_1}$ for $\bu=i_1\cdots i_n\in X$.
\end{Lem}

\begin{proof}
	The lemma follows directly from the fact that  $\{d_{\bu}:\bu\in X\}\subset \Z$ is uniformly discrete and $J, V$ are bounded subsets of ${\mathbb R}^d$.
\end{proof}

\begin{theorem}\label{th-holderequivalence}
	Let $K$ be the integral self-affine set as in \eqref{id-self-affine-set} with $|\det A|=q$. Then the graph $(X, \E)$ is hyperbolic. Moreover, there exists a H\"older bijection $\phi: \partial X \to K$  satisfying the property:
	\begin{equation}\label{eq-holder}
	C^{-1}w(\phi(\xi)-\phi(\eta)) \le \rho_a(\xi,\eta)^\alpha \le Cw(\phi(\xi)-\phi(\eta)), \quad \forall \  \xi\ne \eta \in \partial X
	\end{equation}
		where $\alpha=\frac{\log q}{da}$ and $C>0$ is a constant.
\end{theorem}

\begin{proof}
The proof generalizes the self-similar case by some modifications (see \cite{LaLu13} or \cite{Wa14}).  For any $\bu\in X$ with $|\bu|=n$, let $V=A^n J_\bu$, then $\hbox {diam}_w(V)= \hbox {diam}_w(J):=b$.  By Lemma \ref{lem-WSC}, we have 
$$\#\{\bv\in X: |\bv|=n-1, n \text{ or } n+1, J_\bv\cap J_\bu\ne \emptyset\}\le \gamma(bq^{\frac 1d})+\gamma(b)+\gamma(bq^{-\frac{1}{d}}).$$ 
Hence the graph $(X,\E)$ is of bounded degree.

Suppose $(X,\E)$ is not hyperbolic, by Lemma \ref{lem-criterion of hyperbolicity}, then for any $m>0$, there exists a horizontal geodesic $\pi(\bu_0, \bu_{3m})=[\bu_0,\bu_1,\dots, \bu_{3m}]$ with $\bu_i\in X_n$ for some $n$. Consider $m$-th generation ancestors $\{\bu_0^{-m}, \bu_1^{-m},\dots, \bu_{3m}^{-m}\}$. By the definition of the graph, we have either $\bu_i^{-m}=\bu_{i+1}^{-m}$ or $(\bu_i^{-m},\bu_{i+1}^{-m})\in \E_h$. Then there is a path $p(\bv_0, \bv_1,\dots, \bv_\ell)$ joining $\bu_0^{-m}$ and $\bu_{3m}^{-m}$ where $\bv_0=\bu_0^{-m}, \bv_\ell=\bu_{3m}^{-m}$ and $\bv_i\in\{\bu_0^{-m}, \bu_1^{-m},\dots, \bu_{3m}^{-m}\}$. Without loss of generality, we may assume  $p(\bv_0, \bv_1,\dots, \bv_\ell)$ is the shortest horizontal path joining $\bu_0^{-m}$ and $\bu_{3m}^{-m}$. By the geodesic property of  $\pi(\bu_0, \bu_{3m})$, it is clear that $$\ell\ge |\pi(\bu_0, \bu_{3m})|-2m=m.$$

Now choose $m\ge \gamma$ such that $(3m+1)q^{-m/d}\le 1$, where $\gamma$ is as in Lemma \ref{lem-WSC}. Let $$V'=\bigcup_{i=0}^{3m}J_{\bu_i}.$$ Then $$\hbox {diam}_w(V')\le \sum_{i=0}^{3m}\hbox {diam}_w(J_{\bu_i})= (3m+1)q^{-n/d}\hbox {diam}_w(J)\le q^{(m-n)/d}\hbox {diam}_w(J).$$ 
Note that for each $i$ there exists $j$ such that $\bv_i=\bu_{j}^{-m}$, it follows that $J_{\bu_j}\subset J_{\bv_i}$.  Thus $J_{\bv_i}\cap V'\ne\emptyset$.   Let $V=A^{n-m}V'$. Then $\hbox {diam}_w(V)\le \hbox {diam}_w(J)$.  It follows that $$\#\{\bv\in X:  (J+d_\bv)\cap V\ne\emptyset\}\ge\#\{\bv\in X_{n-m}: J_{\bv}\cap V'\ne\emptyset\}\ge \ell+1>m\ge \gamma,$$ which contradicts Lemma \ref{lem-WSC}. Therefore, $X$ is hyperbolic.

For any geodesic ray $\xi=
\pi[{\mathbf u}_1,{\mathbf u}_2,\ldots]$ of $X$,
we define
$$
\phi(\xi)=\lim_{n\rightarrow \infty} S_{{\mathbf u}_n}(x_0)
$$
for some $x_0\in J$. Then the map is well-defined  and is bijective (see \cite{Wa14}).

To show that $\phi$  is the desired H\"older map,  we let $\xi= \pi[{\mathbf
	u}_0,{\mathbf u}_1,\ldots], \ \eta= \pi[{\mathbf
	v}_0,{\mathbf v}_1,\ldots]$ be any two non-equivalent
geodesic rays in $X$. Then there is a canonical bilateral geodesic $\tau$ joining $\xi$ and $\eta$:
$$
\tau=\pi[\dots,{\mathbf u}_{n+1},{\mathbf u}_n,{\mathbf t}_1,\dots,{\mathbf t}_{\ell},{\mathbf v}_n,{\mathbf v}_{n+1},\dots]
$$
with ${\mathbf u}_n,{\mathbf t}_1,\dots,{\mathbf t}_{\ell}, {\mathbf
	v}_n\in X_n$. It follows that
$$
 w(S_{{\mathbf
		u}_n}(x_0)-S_{{\mathbf v}_n}(x_0))\leq (\ell+2) q^{-n/d} \hbox {diam}_w(J).
$$
By Lemma \ref{lem-criterion of hyperbolicity}, $\ell$ is uniformly bounded. Note that $\phi(\xi)\in
J_{{\mathbf u}_k}$ and $\phi(\eta)\in J_{{\mathbf v}_k}$ for all
$k\geq 0$, hence
$$
w(\phi(\xi)-S_{{\mathbf u}_n}(x_0)), \   w(\phi(\eta)-S_{{\mathbf v}_n}(x_0))
\leq q^{-n/d} \hbox {diam}_w(J).
$$
Using the property $w(x+y)\le \beta\max\{w(x), w(y)\}$ twice,  there exists a constant $C_1>0$ such that
$$
w(\phi(\xi)-\phi(\eta))\le  C_1 q^{-n/d}.
$$
Since $\tau$  is a bilateral canonical
geodesic, we have $|\xi\wedge\eta|=n-(\ell+1)/2$ and $\ell$ is
uniformly bounded. By using $\rho_a(\xi,\eta)=\exp(-a
|\xi\wedge\eta|)$, we see that
$$
w(\phi(\xi)-\phi(\eta))\leq C\rho_a(\xi,\eta)^{\alpha}.
$$

On the other hand,  assume that $\xi \ne \eta$. Since $\tau$ is a
geodesic, it follows that $({\mathbf u}_{n+1},{\mathbf
	v}_{n+1})\notin {\mathcal {E}}_h$, and hence $J_{{\mathbf
		u}_{n+1}}\cap J_{{\mathbf v}_{n+1}}=\emptyset$. By Lemma \ref{h-condition},  there is $k$ (independent of $n$) such that
\begin{equation*}
J_{\bf u} \cap J_{\bf v} = \emptyset \  \  \Rightarrow \  \  \hbox
{dist}_w (J_{\bf ui}, J_{\bf vj}) \geq c q^{-n/d}, \quad \forall \  {\bf
	i},\  {\bf j} \in \Sigma^k.
\end{equation*}
As $\phi(\xi) \in J_{{\bf u}_{n+k+1}}, \ \phi(\xi) \in J_{{\bf v}_{n+k+1}}$,  we have 
$$
w(\phi(\xi)-\phi(\eta))\geq \hbox{dist}_w(J_{{\mathbf u}_{n+k+1}},
J_{{\mathbf v}_{n+k+1}})\geq cq^{-n/d},
$$
and $w(\phi(\xi)-\phi(\eta))\geq c'\rho_a(\xi,\eta)^{\alpha}$
follows by the definition of  $\rho_a$.
\end{proof}

\bigskip

\section{Lipschitz equivalence of self-affine sets}

We first show that the $w$-Hausdorff dimension is an invariant under the $w$-Lipschitz equivalence. 

\begin{Prop}\label{prop-diminvariant}
	If $E\simeq_w F$ then $\dim^w_H E=\dim^w_H F$.
\end{Prop}

\begin{proof}
	The proof is the same as Corollary 2.4 of \cite{F} by replacing the Euclidean norm with pseudo-norm $w$.
\end{proof}

Let $\lambda_1$ and $\lambda_0$ be the maximal and minimal moduli of eigenvalues of $A$ defining the pseudo-norm $w$. There is a relationship between $w$-Lipschitz equivalence and nearly Lipschitz equivalence.

\begin{Prop}\label{prop}
	Suppose $\lambda_0=\lambda_1$. If $E\simeq_w F$ then $E\simeq_n F$.
\end{Prop}

\begin{proof}
	Let  $\lambda=\lambda_0=\lambda_1$ and define a function $h: (0,\lambda-1) \to (0,1)$ by $$h(x)=\frac{\ln(\lambda-x)}{\ln(\lambda +x)}.$$ Obviously $h$ is a  bijection. Hence for any $0<\eta< 1$, we can choose $\epsilon\in (0,\lambda -1)$ such that $\eta=h(\epsilon)$.
	
By taking the bijective map $g(x)=x/|E|$ where $|E|$ is the diameter of $E$ under the Euclidean norm, we have  $E\simeq_n E/|E|$. Similarly   $F\simeq_n F/|F|$. Without loss of generality, we may assume $|E|, |F|\le 1$. Since  $E\simeq_w F$, there  is a bijection $\sigma: E\to F$ satisfying the inequality 
$$C_0^{-1}w(x-y)\leq w(\sigma(x)-\sigma(y))\leq C_0w(x-y),$$ where $C_0$ is a constant.  That together with Proposition \ref{prop-norms} implies
	\begin{eqnarray*}
		\|\sigma(x)-\sigma(y)\|& \le & (C_\epsilon w(\sigma(x)-\sigma(y))^{\frac{d\ln(\lambda-\epsilon)}{\ln q}} \\
		& \le & (C_\epsilon C_0w(x-y))^{\frac{d\ln(\lambda-\epsilon)}{\ln q}} \\
		& \le & \left(C_\epsilon^2 C_0\|x-y\|^{\frac{\ln q}{d \ln (\lambda+\epsilon)}}\right)^{\frac{d\ln(\lambda-\epsilon)}{\ln q}} \\
		& \le & (C_\epsilon^2 C_0)^{\frac{d\ln(\lambda-\epsilon)}{\ln q}}\|x-y\|^{h(\epsilon)}.
	\end{eqnarray*}
The reverse inequality also follows immediately. By letting $\eta=h(\epsilon)$ as the previous argument, we prove that $E\simeq_n F$.
\end{proof}

From now on, we focus on the  IFS $\{S_i\}_{i=1}^N$ in \eqref{id-self-affine-set} and  fix the invariant set $J=\overline{B_\delta}$.  Let $J_k=\bigcup_{\bu\in \Sigma^k}S_{\bu}(J)$ be the $k$-th iteration of $J$ under the IFS, where $S_\bu(J)=A^{-k}(J+d_\bu)$. Obviously the self-affine set $K=\bigcap_{k=1}^\infty J_k$.  Denote by
$$\mathbb{H}_k=J_k+\Z \quad \text{and}\quad\mathbb{H}=K+\Z.$$ Then ${\mathbb H}=\bigcap_{k=1}^\infty {\mathbb H}_k$. 

The following two topological lemmas are straightforward, which were also concerned by Xi and Xiong (\cite{XiXi14}).

\begin{Lem}\label{lem.union.totally dis.}
	The union of finitely many totally disconnected compact subsets of ${\mathbb R}^d$   is also totally disconnected.
\end{Lem}

\begin{proof}
	Let $A_1, A_2$ be totally disconnected compact subsets of ${\mathbb R}^d$, and let $A=A_1\cup A_2$.
	Obviously, if $A_1\cap A_2=\emptyset$ then $A$ is totally disconnected. Otherwise, we need to show that for any $x\in A$ and any open neighborhood $U$ of $x$ there exists an open-closed set $B$ such that $x\in B \subset U$. Let $x\in A_1\cap A_2$ and $U$ is an open neighborhood of $x$ in $A$. Then $U\cap A_i$ is open in $A_i$ for $i=1,2$. Hence there exists an open-closed set $B_i$ in $A_i$ such that $x\in B_i\subset{U\cap A_i}$ where $i=1,2$. It follows that $x\in {B_1\cup B_2} \subset (U\cap A_1) \cup (U\cap A_2)=U$. Since $A_1, A_2$ are compact subsets of $A$, we have $B_1\cup B_2$ is closed in $A$. On the other hand, $U\cap A_i$ is also open in $U$ for $i=1,2$. Then $B_1, B_2$ are open in $U$ and $B_1\cup B_2$ is open in $U$ as well, hence open in $A$. This proves that $B_1\cup B_2$ is open-closed in $A$. The general case follows by induction.
\end{proof}

\begin{Lem}\label{lem-constant}
If the integral self-affine set $K$ is totally disconnected, then there exists $n_0\ge 1$ such that   any component of $\mathbb{H}_{n_0}$  that intersects $\overline{B_\delta}$ must lie in $B_{\delta+1}.$ 
\end{Lem}

\begin{proof}
For each $n$, let $C_n$ be a component of ${\mathbb H}_n$ that intersects  $\overline{B_\delta}$. Suppose $C_n\cap B_{\delta+1}^c\ne\emptyset$. We shall obtain a contraction. Let
 $$U_n=C_n\cap\overline{B_{\delta+1}}\quad \text{and}\quad V_n=C_n\cap B_{\delta+1}^c.$$
 
Let $\Gamma_n$  be a component of $U_n$ that intersects  $\overline{B_\delta}$. We first show that $\Gamma_n$ also intersects the circle $D=\{x\in{\mathbb R}^d: \|x\|=\delta+1\}$. If not,  for any $x\in U_n$ with $\|x\|=\delta+1$, there exist two disjoint closed sets $E_x, F_x$ so that $U_n=E_x\cup F_x$ and $x\in F_x, \Gamma_n\subset E_x$. By compactness, there is a finite subcover $\{F_{x_1}, \dots, F_{x_k}\}$ of $U_n\cap D$. Let $$F=\bigcup_{i=1}^k F_{x_i}, \quad E=\bigcap_{i=1}^kE_{x_i}.$$ Then $U_n=E\cup F$ with disjoint union. Hence $E$ and $F\cup V_n$ form a separation of $C_n$, contradicting the assumption of connectedness of $C_n$.

Under the Hausdorff metric $d_H$, we know that there is a convergent subsequence of $\{\Gamma_n\}_n$. Without loss of generality, we may assume $\Gamma_n \to \Gamma$. Then $\Gamma$ is a connected closed set that intersects both $\overline{B_\delta}$ and $D$. Indeed,  if $\Gamma$ is not connected, then there is a separation $\Gamma=A\cup B$ where $A, B$ are nonempty closed sets and thus are compact, and 
$$\epsilon:=\inf\{\|a-b\|: a\in A, \ b\in B\}>0.$$ Let $\Gamma_n$ be a component such that $d_H(\Gamma_n, \Gamma)< \epsilon/3.$ Then $\Gamma_n$ is contained in an $\epsilon/3$-neighborhood of $A$ and $B$, and $\Gamma_n$ cannot be connected. That is ridiculous.

Since $\Gamma_n\subset {\mathbb H}_n\cap \overline{B_{\delta+1}}$ and $ {\mathbb H}_n\cap \overline{B_{\delta+1}} \to  {\mathbb H}\cap \overline{B_{\delta+1}}$ under the metric $d_H$. It follows that $\Gamma\subset  {\mathbb H}\cap \overline{B_{\delta+1}}$. This contradicts the fact that ${\mathbb H}\cap \overline{B_{\delta+1}}$ is totally disconnected by Lemma \ref{lem.union.totally dis.}.
\end{proof}

If a hyperbolic graph  $(X,\E)$ induced by an IFS is of bounded degree, then Theorem 5.5 of \cite{L17} shows that {\it $\partial X$ (or the fractal $K$) is totally disconnected if and only if the sizes of horizontal components in $(X,\E)$ are uniformly bounded.}   Under the present  setting, the statement can be strengthened as the following version. 

\begin{Lem}\label{th-simple}
The integral self-affine set $K$ is totally disconnected if and only if the graph $(X,\E)$ is simple.
\end{Lem}

\begin{proof}
Theorem \ref{th-holderequivalence} says that $(X,\E)$ is a hyperbolic graph with bounded degree. If it is simple, then there are only finitely many equivalence classes of horizontal components, hence $K$ is totally disconnected.

Conversely suppose $K$ is totally disconnected. Let $n_0$ be a constant in Lemma \ref{lem-constant}. Obviously there are finite equivalence classes of horizontal components in $\bigcup_{j=1}^{n_0} X_j$.  Let $T\subset \bigcup_{j>{n_0}} X_j$ be a horizontal component. We may assume $T=\{\bu_1,\dots, \bu_k\}\subset X_n$ for $n>n_0$. Then $J_T:=\bigcup_{j=1}^kS_{\bu_j}(J)$ is connected. Decompose each word $\bu_j$ by $\bu_j=\bu^1_j\bu^2_j$  where $\bu^1_j\in X_{n-n_0}$ and $\bu^2_j\in X_{n_0}$. We can write $S_{\bu_j}=S_{\bu^1_j}\circ S_{\bu^2_j}$. Hence
\begin{eqnarray*}
A^{n-n_0}J_T &=& \bigcup_{j=1}^k A^{n-n_0}S_{\bu_j}(J) \\
&=&\bigcup_{i=1}^k A^{n-n_0}S_{\bu^1_j}\circ S_{\bu^2_j}(J) \\
&=&\bigcup_{i=1}^k (S_{\bu^2_j}(J)+d_{\bu^1_j})\\
&\subset& \mathbb{H}_{n_0}.
\end{eqnarray*}
Choose  $\bd\in \Z$ such that $(A^{n-n_0}J_T-\bd)\cap \overline{B_\delta}\ne\emptyset.$ 
As $A^{n-n_0}J_T-\bd\subset \mathbb{H}_{n_0}$, by Lemma \ref{lem-constant}, we have $A^{n-n_0}J_T-\bd\subset B_{\delta+1}.$ 
	Hence $$A^{n}J_T-A^{n_0}\bd=\bigcup_{j=1}^k(J+d_{\bu_j})-A^{n_0}\bd\subset A^{n_0}B_{\delta+1}.$$
	
Since $k$ is uniformly bounded and $d_{\bu_j}-A^{n_0}\bd\in \Z$, there are finitely many connected sets $\bigcup_{j=1}^k(J+d_{\bu_j})$ in $A^{n_0}B_{\delta+1}$, up to translation. Therefore, the equivalence classes of horizontal components of $X$ is finite by definition.
\end{proof}

\begin{theorem}\label{th-ifandonlyif}
	Let $K=A^{-1}(K+{\mathcal D}_1)$ and $K'=A^{-1}(K'+{\mathcal D}_2)$ be two  integral self-affine sets. Suppose both the IFSs satisfy the OSC and $K, K'$ are totally disconnected. Then $K\simeq_w K'$ if and only if $\#{\mathcal D}_1=\#{\mathcal D}_2$.
\end{theorem}

\begin{proof}
	If $K\simeq_w K'$, by Proposition \ref{prop-diminvariant}, then  $\dim^w_H K =\dim^w_H K'$. It follows from Theorem \ref{th-dimformular} that $\#{\mathcal D}_1=\#{\mathcal D}_2$.
	
	Conversely, let $\#{\mathcal D}_1=\#{\mathcal D}_2=N$ and let $(X, \E), (Y, \E')$ be the hyperbolic graphs induced on $K, K'$ respectively. Since the OSC holds, both $(X, \E), (Y, \E')$  are $N$-ary augmented trees satisfying Definition \ref{Def'}.  From Theorem \ref{th1.1} and Lemma \ref{th-simple}, it yields that 
	$$\partial (X, \E) \simeq \partial (X, \E_v)=\partial (Y, \E_v')\simeq \partial (Y, \E').$$
	
	Let $\varphi: \partial X  \to \partial Y$ be a bi-Lipschitz map. By Theorem \ref{th-holderequivalence}, there exist two bijections $\phi_1:\partial X\to K$ and $\phi_2:\partial Y\to K'$ satisfying \eqref{eq-holder} with constants $C_1,C_2$, respectively.  Now we define $\sigma: K \to K'$ as
	$$
	\sigma = \phi_2\circ \varphi\circ \phi_1^{-1}.
	$$
	Then
	$$
	\begin {aligned}
	w(\sigma(x) -\sigma (y)) & \leq
	C_2\ \rho_a(\varphi\circ\phi_1^{-1}(x),\varphi\circ\phi_1^{-1}(y))^\alpha\\
	&\leq C_2C_0^\alpha\ \rho_a(\phi_1^{-1}(x),\phi_1^{-1}(y))^\alpha\\
	&\leq
	C_2C_0^\alpha C_1 w(x-y).
	\end{aligned}
	$$
	Let $C' = C_2C_0^\alpha C_1$, then $w(\sigma (x) - \sigma (y)) \leq C'w(x -y)$. Moreover,  ${C'}^{-1} w(x-y) \leq w(\sigma (x) - \sigma (y))$ follows from another inequality of \eqref{eq-holder}. Therefore  $K\simeq_w K'$.
\end{proof}

\begin{Cor}\label{cor-iff}
	Under the assumption of Theorem \ref{th-ifandonlyif}. If the eigenvalues of $A$ have  equal  moduli.  Then $K\simeq_n K'$ if and only if $\#{\mathcal D}_1=\#{\mathcal D}_2$.
\end{Cor}

\begin{proof}
	The if part follows from Theorem \ref{th-ifandonlyif} and Proposition \ref{prop}. For the only if part, if  $K\simeq_n K'$, then $\dim_H K= \dim_H K'$. Since the eigenvalues of $A$ have  equal  moduli, Proposition \ref{prop-norms} implies  that $\dim_H^w K= \dim_H^w K'$.  Therefore,    $\#{\mathcal D}_1=\#{\mathcal D}_2$ by Theorem \ref{th-dimformular}.
\end{proof}

If $\mathcal D\subset {\mathbb Z}^d$ is a set of coset representatives of $ {\mathbb Z}^d/{A {\mathbb Z}^d}$, i.e., $d_i+A {\mathbb Z}^d\ne d_j+A {\mathbb Z}^d$ for distinct $d_i, d_j\in {\mathcal D}$. It is well-known that the pair $(A, {\mathcal D})$ satisfies the OSC. The following corollary is immediate.

\begin{Cor}\label{cor}
	Let $K=A^{-1}(K+{\mathcal D}_1)$ and $K'=A^{-1}(K'+{\mathcal D}_2)$ be two  totally disconnected integral self-affine sets where  ${\mathcal D}_1, {\mathcal D}_2$ are sets of coset representatives of $ {\mathbb Z}^d/{A {\mathbb Z}^d}$.  Then $K\simeq_w K'$ if and only if $\#{\mathcal D}_1=\#{\mathcal D}_2$.
\end{Cor}

Let $A=\left[\begin{array}{cc}
m & 0\\
0 & n
\end{array} \right]$ be an expanding matrix where $m,n\ge 2$ are integers, and let ${\mathcal D}\subset \{0,1,\dots, m-1\}\times \{0,1,\dots, n-1\}$ be a digit set. Then the associated self-affine set $K=A^{-1}(K+\D)$ is the well-known McMullen-Bedford set (\cite{Be84},\cite{Mc84}). The standard Hausdorff (or box) dimension formula of McMullen-Bedford set has been obtained.

From Corollary \ref{cor}, it can be seen  that 	 two totally disconnected McMullen-Bedford sets are $w$-Lipschitz equivalent if and only if their digit sets have equal cardinality. However, the two $w$-Lipschitz equivalent McMullen-Bedford sets may be not Lipschitz equivalent under the Euclidean norm, as they maybe have distinct Hausdorff dimensions.

\begin{Exa}\label{exa}
{\rm 
Let $A=\left[\begin{array}{cc}
3 & 0\\
0 & 4
\end{array} \right]$ and let three digit sets be as follows
\begin{eqnarray*}
	\D_1 &=& \{(0,0),(0,1),(0,2),(1,2),(1,3),(2,0),(2,1)\}, \\
	\D_2 &=& \{(0,0),(0,1),(0,3),(1,2),(1,3),(2,0),(2,1)\}, \\ 
	\D_3 &=& \{(0,0),(0,1),(0,2),(1,2),(2,0),(2,1),(2,3)\}.
\end{eqnarray*} 
Let  $K, K', K''$ be  the associated  McMullen-Bedford sets respectively (see Figure \ref{McMullen sets}).  A dimension formula (\cite{Mc84},\cite{F}) yields that $$\dim_H K = \dim_H K'=\log_3(3^{\log_4 3}+2^{1+\log_4 3}), \  \dim_H K''=\log_3(2\cdot 3^{\log_4 3}+1)$$ which are different. Hence $K, K''$ and  $K', K''$ are both not Lipschitz equivalent under the Euclidean norm. It is also not clear that if $K\simeq K'$ even $\dim_H K = \dim_HK'$.

On the other hand, by using a criterion for integral self-affine sets to be totally disconnected (\cite{Zh17}), it can be verified that $K, K', K''$ are all totally disconnected. Therefore, $K\simeq_w K'\simeq_w K''$ by Corollary \ref{cor}.
}
\end{Exa}

\bigskip
\noindent {\bf Acknowledgements:} The author gratefully acknowledges the support of K. C. Wong Education Foundation and DAAD. He also would like to thank Prof. Martina Z\"ahle for valuable discussions.


\bigskip

\begin{thebibliography}{99}

\bibliographystyle{ieee}
\addcontentsline{toc}{chapter}{Bibliography}
\bibitem{Be84} T. Bedford, {\it Crinkly curves, Markov partitions and box dimension in self-similar sets}, Ph.D. thesis, University of Warwick, 1984.

\bibitem{CoPi88} D. Cooper and T. Pignataro, {\it On the shape of Cantor sets}, J. Differential Geom.  28 (1988)  203-21.

\bibitem{DaSe97} G. David and S. Semmes, {\it Fractured Fractals and Broken Dreams: Self-Similar Geometry through Metric and Measure}, Oxford Univ. Press, 1997.


\bibitem{DLL15} G.T. Deng, K.S. Lau and J.J. Luo, {\it  Lipschitz equivalence of self-similar sets and hyperbolic boundaries II},  J. Fractal Geom. 2 (2015), no. 1, 53-79.


\bibitem{F} K.J. Falconer, {\it Fractal geometry. Mathematical foundations and applications}, Second edition. John Wiley \& Sons, Inc., Hoboken, NJ, 2003. xxviii+337 pp.

\bibitem{FaMa89} K.J. Falconer and D.T. Marsh, {\it Classification of quasi-circles by Hausdorff dimension}, Nonlinearity (1989) 2(3):489-493.

\bibitem{FaMa92} K.J. Falconer and D.T. Marsh, {\it On the Lipschitz equivalence of Cantor sets}, Mathematika  39 (1992) 223-233.



\bibitem{Gr87} M. Gromov, {\it Hyperbolic groups}, MSRI Publications 8, Springer Verlag (1987), 75-263.

\bibitem{HeLa08} X.G. He and K.S. Lau, {\it On a generalized dimension of self-affine fractals}, Math. Nachr. 281 (2008) No.8, 1142-1158.


\bibitem{Ka03} V.A. Kaimanovich, {\it Random walks on Sierpinski graphs: Hyperbolicity and stochastic homogenization}, Fractals in Graz 2001, Trends Math., Birkhuser, Basel (2003), 145-183.


\bibitem{LaNgRa01}  K. S. Lau, S. M. Ngai, and H. Rao, {\it Iterated function systems with overlaps and self-similar measures}, J. London Math. Soc. (2) 63, (2001), 99-116.

\bibitem{LaWa09} K.S. Lau and X.Y. Wang, {\it Self-similar sets as hyperbolic boundaries}, Indiana U. Math. J.  58 (2009), 1777-1795.


\bibitem{LaWa16} K.S. Lau and X.Y. Wang, {\it On hyperbolic graphs induced by iterated function systems}, Adv. Math, to appear.

\bibitem{LlMa10}  M. Llorente and P. Mattila, {\it Lipschitz equivalence of subsets of self-conformal sets}, Nonlinearity  23 (2010) 875-882.


\bibitem{L17} J.J. Luo, {\it Self-similar sets, simple augmented trees, and their Lipschitz equivalence}, preprint (2016).

\bibitem{LuYa10} J.J. Luo and Y.M. Yang, {\it On single-matrix graph-directed iterated function systems},  J. Math. Anal. Appl. 372 (2010), no.1, 8-18.

\bibitem{LaLu13}J.J. Luo and K.S. Lau, {\it Lipschitz equivalence of self-similar sets and hyperbolic boundaries}, Adv. Math.  235 (2013), 555-579.


\bibitem{Mc84} C. McMullen, {\it The Hausdorff dimension of general Sierpi\'nski carpets}, Naygoya Math. J., 96 (1984), 1-9.


\bibitem{RaRuXi06} H. Rao, H.J. Ruan and L.-F. Xi,  {\it Lipschitz equivalence of self-similar sets}, CR Acad. Sci. Paris, Ser. I {\bf 342} (2006), 191-196.

\bibitem{RaRuWa10} H. Rao, H.J. Ruan and Y. Wang,  {\it Lipschitz equivalence of Cantor sets and algebraic properties of contraction ratios}, Trans. Amer. Math. Soc.  364 (2012), 1109-1126.

\bibitem{RaZh15} H. Rao and Y. Zhang, {\it Higher dimensional Frobenius problem and Lipschitz equivalence of Cantor sets}, J. Math. Pures Appl. 104 (2015) 868-881.


\bibitem{Sc94} A. Schief, {\it Separation properties for self-similar sets}, Proc. Amer. Math. Soc. 122 (1994)  111-115.

\bibitem{Wa14} X.Y. Wang, {\it Graphs induced by iterated function systems}, Math. Z. (2014) 277:829-845.

\bibitem{Wa99} Y. Wang, {\it Self-affine tiles}, Advances in wavelets (Hong Kong, 1997), 261-282, Springer, Singapore, 1999. 


\bibitem{Wo00} W. Woess, {\it Random Walks on Infinite Graphs and Groups}, Cambridge Tracts in Mathematics vol.138, Cambridge University Press, Cambridge
2000.

\bibitem{Xi04} L.F. Xi {\it Lipschitz equivalence of self-conformal sets}, J. London Math. Soc.(2) (2004) 70(2): 369-382.

\bibitem{XiXi10} L.-F. Xi and Y. Xiong, {\it Self-similar sets with initial cubic patterns}, \textit{CR Acad. Sci. Paris, Ser. I}   348 (2010) 15-20.

\bibitem{XiXi13} L.-F. Xi and Y. Xiong,  Lipschitz equivalence class, ideal class and the Gauss class number problem, arXiv:1304.0103.

\bibitem{XiXi14} L.-F. Xi and Y. Xiong, {\it Rigidity theorem of graph-directed fractals}, arXiv:1308.3143.

\bibitem{Zh17} Y. Zhang, {\it Topological structure of a class of planar self-affine sets}, J. Math. Anal. Appl. 446 (2017) 1524-1536.
\end{thebibliography}
\end{document}